\documentclass[12pt]{amsart}
\usepackage{amssymb, amsmath}
\DeclareMathOperator{\Span}{span}

\DeclareMathOperator{\co}{co}
\DeclareMathOperator{\face}{face}

\def\C{\mathbb{C}}
\def\R{\mathbb{R}}
\def\B{\mathcal{B}}
\DeclareMathOperator{\RR}{\,R\,}

\newtheorem{theorem}{Theorem} 
\newtheorem{lemma}[theorem]{Lemma} 
\newtheorem{corollary}[theorem]{Corollary}

\theoremstyle{remark}
\newtheorem*{notation}{Notation}
\newtheorem*{definition}{Definition}

\newif\ifproofing
\newcommand{\note}[1]{\ifproofing\marginpar{#1}\fi}

\newenvironment{lem}[2]{
\note{#1}\begin{lemma}\label{#1}#2
}
{\end{lemma}}

\newenvironment{thm}[2]{
\note{#1}\begin{theorem}#2\label{#1}
}
{\end{theorem}}

\newenvironment{cor}[2]{
\note{#1}\begin{corollary}\label{#1}#2
}
{\end{corollary}}

\newenvironment{eqn}[2]{
\note{#1}\begin{equation}\label{#1}#2
}
{\end{equation}}

\begin{document}

\title[Unique decompositions of separable states]{Unique decompositions, faces, and automorphisms of separable states}
\author{Erik Alfsen}
\address{Mathematics Department, University of Oslo, Blindern 1053, Oslo, Norway}
\author{Fred Shultz}
\address{Mathematics Department, Wellesley College, Wellesley, Massachusetts 02481, USA}
\date{September 17, 2009}
\begin{abstract}
Let $S_k$ be the set of separable states on $\B(\C^m \otimes \C^n)$ admitting a representation as a convex combination of $k$ pure product states, or fewer.  If $m>1, n> 1$, and  $k \le \max{(m,n)}$,  we show that $S_k$ admits a subset $V_k$ such that $V_k$ is dense and open in $S_k$, and such that each state in $V_k$ has a unique decomposition as a convex combination of pure product states, and we describe  all possible convex decompositions for a set of separable states that properly contains $V_k$. In both cases we describe the associated faces of the space of separable states, which in the first case are simplexes, and in the second case are direct convex sums of faces that are isomorphic to state spaces of full matrix algebras. As an application of these results, we characterize all affine automorphisms of the convex set of separable states, and all automorphisms of the state space of $\B(\C^m \otimes \C^n)$ that preserve entanglement and separability.
\end{abstract}

\maketitle

\section{Introduction}

 A state on the algebra $\B(\C^m \otimes \C^n)$ of linear operators is separable if it is a convex combination of product states.  States that are not separable are said to be entangled, and are of substantial interest in quantum information theory.  Easily applied conditions for separability are known only for special cases, e.g., if $m = n = 2$, then a state is separable iff its associated density matrix has positive partial transpose, cf. \cite{Peres, Horodeckis}.  Other necessary and sufficient conditions are known, e.g.  \cite{Horodeckis}, but are not easily applied in practice.  An open question of great interest is to find a simple necessary and sufficient condition for a state to be separable. 
 
A product state $\omega \otimes \tau$ is a pure state iff $\omega$ and $\tau$ are pure states. Thus a separable state is precisely one that admits a representation as a convex combination of pure product states.
 It is natural to ask the extent to which this decomposition is unique.  That is the main topic of this article.

For the full state space $K$ of $\B(\C^m \otimes \C^n)$ each non-extreme point can be decomposed into extreme points 
in many different ways. But for the space $S$ of separable states the situation is totally different. While non-extreme points
with many different decompositions exist (and are easy to find) in $S$ as well as in $K$, 
there are in $S$ also plenty of points for which the decomposition is unique. 

DiVincenzo, Terhal, and  Thapliyal \cite{DiV}   defined the \emph{optimal ensemble cardinality} of a separable state $\rho$ to be $k$ if   $k$ is the minimal number of pure product states required for a convex decomposition of $\rho$. Lockhart \cite{Lockhart} used the term ``optimal ensemble length" for the same notion. For brevity, we will  simply call this number the \emph{length} of $\rho$, and we denote the set of separable states of length at most $k$ by $S_k$. 
We show in Theorem \ref{cor4} that for $m>1, n> 1$ and $k \le \max(m,n)$, the set $S_k$ has a  subset  $V_k$ which is  dense and open in $S_k$, with each $\sigma \in V_k$ admitting a unique decomposition into pure product states. Actually,
we exhibit such a set $V_k$ consisting of states with the property that each generates 
a face of $S$ which  is a simplex, from which the uniqueness follows.  

We remark that the sets $V_k$ are open and dense in the relative topology on $S_k$, but are not open or dense in $S$ or $K$ if $mn > 1$. (See the remarks after Theorem \ref{cor4}).   Indeed it would be surprising if a subset of low rank separable states were open and dense in the set of all states of that rank, since low rank states are almost surely entangled \cite{RuskaiWerner, WalgateScott}, and in general  $S$ has measure  which is a decreasingly small fraction of the measure of $K$ as $m, n$ increase, cf. \cite{AubrunSzarek, Szarek}.

While dimensions are too high to be able to accurately visualize the above results, the  reader may be curious about the relationship  to the well known tetrahedron/octahedron picture for $m = n = 2$, cf. \cite{HorodeckiTetra}. In that picture, 
there is a subset $\mathcal{T}$ of states which is a tetrahedron, and which has the property that for every state $\rho$ which restricts to the normalized trace on $\B(\C^2) \otimes I$ and on $I \otimes B(\C^2)$, there are unitaries $U$ and $V$ such that $(U\otimes V)^*\rho(U\otimes V) \in \mathcal{T}$.
The midpoints of the six edges of this tetrahedron are the vertices of an octahedron that consists of the separable states in $\mathcal{T}$.  Each vertex of the octahedron is  a convex combination of two distinct pure product states (which of course are not in $\mathcal{T}$), cf. \cite[eqn. (63)]{Ovrum}. In fact,  the vertices are the only states in the octahedron of length $\le  \max(m,n)= 2$.

 It can be checked (e.g., by applying our Corollary \ref{cor3.9}) that the decomposition of each of these vertices into pure product states is unique.   Each state in the interior of this octahedron has rank $4 = mn$, so is an interior point of the full state space $K$, hence has a non-unique convex decomposition into pure product states (see the remarks after Theorem \ref{cor4}.) The tetrahedron also arises as a parameterization for a set of unital completely positive trace preserving maps from $M_2(\C)$ to $M_2(\C)$, with the octahedron consisting of the  entanglement breaking maps in this set, cf.  \cite[Appendix B]{RuskaiKing}, \cite[Thm. 4]{Ruskai}, and \cite[Fig. 2]{RuskaiWerner}.

We also define a broader class of states that we show have a unique decomposition 
as a convex combination of product states $\rho_i\otimes\sigma_i$ that are not necessarily
pure, but with the property that each of them generates  a face of $S$ which is also a face of $K$ and 
is affinely  isomorphic  to the state space of $\B(\C^{p_i})$ for a suitable $p_i$. From 
this it  follows that the ambiguity in decompositions for a given state in this class is restricted to 
the ambiguity in decompositions for points in the state space of the matrix algebras 
$\B(\C^{p_i})$.
For a complete description of the possible decompositions of a state on $\B(\C^{p})$,
see \cite{Kirkpatrick, Schrodinger,Wootters}. 

We use our results on the facial structure of $S$ to show that every affine automorphism of the space $S$ of separable states on $\B(\C^m \otimes \C^n)$ is given by a composition of the duals of the maps that are (i) conjugation by local unitaries (i.e., unitaries of the form $U_1 \otimes U_2$)  (ii) the two partial transpose maps, or (iii) the swap automorphism that takes $A\otimes B$ to $B \otimes A$ (if $m = n$).  A consequence is a description of the affine automorphisms $\Phi$ of the state space  such that $\Phi$  preserves entanglement and separability.  

There is related work of Hulpke  et al \cite{Hulpke}. They say a linear map $L$ on $\C^m \otimes \C^n$ preserves \emph{qualitative entanglement} if $L$ sends separable (i.e., product) vectors to product vectors, and entangled vectors to entangled vectors.  They  show that a linear map $L$ preserves qualitative entanglement of vectors on $\C^m \otimes \C^n$ iff $L$ is a local operator (i.e. one of the form $L_1 \otimes L_2$), or if $L$ is a local operator composed with the swap map that takes $x\otimes y$ to $y \otimes x$.  They then show that if $L$ preserves a certain \emph{quantitative} measure of entanglement, then $L$ must be a local unitary. 

We thank Mary Beth Ruskai for helpful comments and references.

\section{Background: states on $\B(\C^n)$}

We review basic facts about states on $\B(\C^n)$,  and  develop some facts about the relationship of independence of vectors $x$ in $\C^n$ and of the corresponding vector states $\omega_x$. In the following sections we will specialize to the case of interest: separable states.

\begin{notation}
 If $x$ is a vector in any vector space, $[x]$ denotes the subspace generated by $x$.   $\C^n$ denotes the set of $n$-tuples of complex numbers viewed as an inner product space with the usual inner product (linear in the first factor).
$\B(\C^n)$ denotes the linear transformations from $ \C^n$ into itself.  For each unit vector $x \in \C^n $, we denote the associated vector state by $\omega_x $, so that $\omega_x(A) = (Ax, x)$. The convex set of states on $\B(\C^n)$ will be denoted by $K_n$.
\end{notation}

We recall that faces of the state space $K_n$ of $\B(\C^n)$ are in 1-1 correspondence with the projections in $\B(\C^n)$, and thus with the subspaces of $\C^n$ that are the ranges of these projections.  If $Q$ is a projection in $\B(\C^n)$, then the associated face $F_Q$ of $K_n$ consists of all states taking the value $1$ on $Q$.  The restriction map is an affine isomorphism from $F_Q$ onto the state space of $Q\B(\C^n )Q \cong \B(Q(\C^n))$. Thus $F_Q$ is affinely isomorphic to the state space of $\B(L)$, where $L = Q(\C^n)$. The set of extreme points of $K_n$ are the vector states, and it follows that the extreme points of $F_Q$ are the vector states $\omega _x$ with $x$ in the range of $Q$, and $F_Q$ is the convex hull of these vector states.  For background, see \cite[Chapter 4]{Alfsen-Shultz}

\begin{definition}
Recall that a convex set $C$ is said to be the \emph{direct convex sum} of a collection of convex subsets $C_1, \ldots, C_p$ if each point $\omega \in C$ can be uniquely expressed as a convex combination 
\begin{eqn}
{unique}\omega = \sum_{i \in I} \lambda_i \omega_i
\end{eqn}
where $I \subset \{1, \ldots, p\}$, $\lambda_i > 0$ for all $i \in I$, $\omega_i \in C_i$ for all $i \in I$, and $\sum_{i \in I} \lambda_i = 1$.
\end{definition}

If $C$ is a convex subset of a real linear space and is located on an affine hyperplane which does not contain the origin (as is the case for our state spaces), then it is easily seen that $C$ is the direct convex sum of convex subsets $C_1, \ldots, C_p$ iff the span of $C$ is the direct sum of the real subspaces spanned by $C_1, \ldots, C_p$.

 A finite dimensional convex set is a \emph{simplex} if it is the direct convex sum of a finite set of points. If the affine span of the points does not contain the origin, then their convex hull is a simplex iff the points are linearly independent (over $\R$).

\begin{lem}{directsumlemma}{Let $L$ be a subspace of $\C^n$ and suppose that $L$ is the direct sum of subspaces $L_1, \dots, L_p$. Let $F_1, \ldots, F_p$ be the corresponding faces of the state space of $\B(\C^n)$. Then the convex hull of $F_1, \ldots, F_p$ is the direct convex sum of those faces.  In particular, if $x_1, \ldots, x_p$ are linearly independent unit vectors, then the corresponding vector states are linearly independent and the convex hull of the corresponding vector states is a simplex.}
\end{lem}

\begin{proof}

Let $I \subset \{0, \ldots, p\}$, and suppose $\{\omega_i \mid i \in I\}$ are nonzero functionals on $\B(\C^n)$ with $\omega_i \in \Span_{\R} F_i$ for each $i$.  
To prove independence  of $\{\omega_i \mid i \in I\}$, suppose that for scalars $\{\gamma_i\}_{i \in I}$ we have 
\begin{eqn}{dependence}\sum_{i \in I} \gamma_i \omega_i = 0.\end{eqn}
 Let $L_0$ be the orthogonal complement of $L$.  Then $\C^n$ as a linear space is the direct sum of  $L_0, L_1, \ldots, L_p$.  
 
 For each $i \in I$, let $P_i$ be the projection associated with $F_i$.  Then we can find $A_i \in P_i\B(\C^n)P_i$ such that $\omega_i(A_i) \not= 0$. Let $B_i \in \B(\C^n)$ be an operator  such that $B_i$ is zero on  $\sum_{j \not= i} L_j$, and such that $\omega_i(B_i) \not= 0$ (e.g., set $B_i= A_i$ on $L_i$).
 If $x \in L_j$ and $j \not= i$, then $\omega_x(B_i) = (B_ix,x) = 0$.  Since every state in $F_j$ is a convex combination of vector states $\omega_x$ with $x \in L_j$, then $\omega_j(B_i) = 0$ if $j \not= i$. 
 
 Now apply both sides of \eqref{dependence} to $B_k$ to conclude that $\gamma_k\omega_k(B_k) = 0$ for all $k \in I$, so $\gamma_k = 0$ for all $k \in I$.  
 Thus the set of vectors $\omega_1, \ldots, \omega_p$  is independent. We conclude that $\co(F_1, \ldots, F_p)$ is the direct convex sum of $F_1, \ldots, F_p$.

If $x_1, \ldots, x_p$ are linearly independent unit vectors, applying the result above with $F_i = \{\omega_{x_i}\}$ shows that the convex hull of the vector states $\omega_{x_i}$ is a simplex. Hence the set $\{\omega_{x_1}, \ldots, \omega_{x_p}\}$ is linearly independent.
\end{proof}

Note that the converses of the statements above are not true.  For example, while no set of more than two vectors in $\C^2$ is independent, it is easy to find a set of three linearly independent vector states on $\B(\C^2)$.

\section{Uniqueness of decompositions of separable states}

We now turn to faces of the set of separable states on $\B(\C^m\otimes\C^n)$, and to the question of uniqueness of convex decompositions of such states. We identify $\B(\C^m\otimes\C^n)$ with $\B(\C^m) \otimes \B(\C^n)$ by $(A\otimes B)(x\otimes y) = Ax \otimes By$. We denote the convex set of all states on $\B(\C^m\otimes\C^n)$ by $K$, and the convex set of all separable states by $S$.

\begin{lem}{productvectorlemma}
{
 Let $e_1, e_2, \ldots, e_p$  and $f_1, f_2, \ldots, f_p$ be unit vectors in $\C^m$ and $\C^n$ respectively.  We assume that $f_1, f_2, \ldots, f_p$ are linearly independent. If $e \in \C^m$ and $f\in \C^n$ are unit vectors such that $e\otimes f$ is in the linear span of $\{ e_i \otimes f_i \mid 1\le i \le p \}$, then there is an index $j$ such that $[e] = [e_j]$ and such that $f$ is in the span of those $f_i$ such that $[e_i] = [e_j]$. In the special case where  $[e_1], \ldots, [e_p]$  are distinct, then $[e] = [e_j]$ and $[f] = [f_j]$  for some index $j$, and $\{ e_i \otimes f_i \mid 1\le i \le p \}$ is independent.}
\end{lem}

\begin{proof}   Extend $f_1 ,\ldots, f_p$ to a basis $f_1, \ldots, f_n$ of $\C^n$, and let $\widehat f_1, \ldots, \widehat f_n$ be the dual basis.  For $1 \le k \le n$, let $T_k : \C^m \otimes \C^n \to \C^m$ be the linear map such that $T_k(x\otimes y) = \widehat f_k(y) x$ for $x \in \C^m$, $y \in \C^n$. 

Suppose that the product vector $e \otimes f$ is a linear combination 
\begin{equation} e\otimes f = \sum_{i=1}^p \alpha_i e_i \otimes f_i.\label{comb}
\end{equation}
 For $j > p$, applying $T_j$ to both sides of \eqref{comb} gives  $\widehat f_j(f)e = 0$, so $\widehat f_j(f) = 0$ for all such $j$. 
Now if $1 \le j \le p$, applying  $T_j$ to both sides of \eqref{comb} gives
\begin{equation}
 \widehat f_j(f) e = 
\alpha_j e_j.
\label{tj}\end{equation} 
 Since $\widehat f_j(f) $ can't be zero for all $j$, then $e$ is a multiple of some $e_j$. Fix such an index $j$. If $1 \le i \le p$ and  $[e_i] \not= [e_j]$, then $e_i$ can't be a multiple of $e$, so   $\widehat f_i(f) e = \alpha_i e_i$  implies  $\alpha_i = 0$, and then also $\widehat f_i(f) = 0$.  We have shown that $\widehat f_i(f) = 0$ if $i > p$, or if $i \le p$ and $[e_i] \not= [e_j]$. It follows that $f$ is in the linear span of those $f_i$ such that $[e_i] = [e_j]$.

If it also happens that  $[e_1], \ldots, [e_p]$  are distinct, and $[e] = [e_j]$, then  $[f] = [f_j]$.  Suppose now that $\sum_i \alpha_i e_i \otimes f_i = 0$.  If $\alpha_k \not= 0$, then $e_k \otimes f_k$ is a linear combination of $\{e_i \otimes f_i \mid i \not= k\}$.  Thus by the conclusion just reached, we must have $[e_k] =[e_i]$ for some $i \not = k$, contrary to the hypothesis that  $[e_1], \ldots, [e_p]$  are distinct.  We conclude that $\alpha_k = 0$ for all $k$, and we have shown that  $\{ e_i \otimes f_i \mid 1\le i \le p \}$ is independent.
\end{proof}

\begin{lem}{faceofKcase}{Let $e_1, \ldots, e_p \in \C^m$ and $f_1, \ldots, f_p \in \C^n$ be unit vectors.  If $[e_1] = [e_2] = \ldots = [e_p]$, then the face $F$ of $S$ generated by the states $\{\omega_{e_i\otimes f_i} \mid 1 \le i \le p\}$ is also a face of $K$, and this face of $K$ is associated with the subspace $L = e_1 \otimes \Span\{f_1, \ldots, f_p\}$ of $\C^m \otimes \C^n$, and $F$ is affinely isomorphic to the state space of $\B(L)$.}
\end{lem}

\begin{proof} 
Let $G$ be the face of $K$ which is associated with the subspace $L$ of $\C^m \otimes C^n$.  By assumption each $e_i$ is a multiple of $e_1$, so that
$$L = \Span\{e_1\otimes f_i \mid 1 \le i \le p\} = \Span\{e_i\otimes f_i \mid 1 \le i \le p\}.$$
Hence $G$ is the face of $K$ generated by $\{\omega_{e_i \otimes f_i} \mid 1 \le i \le p\}$.

We would like to show $G = F$. For brevity we denote the convex hull of the set $\{\omega_{e_i \otimes f_i} \mid 1 \le i \le p\}$ by $C$, and observe that $G$ and $F$ are the faces of $K$ and $S$ respectively generated by $C$.  It follows easily from the definition of a face that the face generated by the convex set $C$ in either one of the two convex sets $S$ or $K$ consists of all points $\rho$ in $S$ or $K$ respectively which satisfy an equation
\begin{equation}\label{rho}
\omega = \lambda \rho + (1-\lambda)\sigma
\end{equation}
where $0 < \lambda < 1$, $\omega \in C$, and where $\sigma$ is in $S$ or $K$ respectively.  It follows that $F = \face_S(C) \subset \face_K(C) = G$.

Since each vector in $L$ is a product vector, the extreme points of $G$ are pure product states, so $G \subset S$.  If $\rho$ is in the face $G$ of $K$ generated by $C$, then we can find  $\sigma \in K$ and $\omega \in C$  such that  \eqref{rho} holds. Then $\sigma$ is also in $G \subset S$, so both $\rho$ and $\sigma $ are in $S$. Hence $\rho$ is in the face $F$ of $S$ generated by $C$. Thus $G \subset F$, and so $F= G$ follows.

\end{proof}

So far we have considered collections of product vectors $\{e_i \otimes f_i\}$ with $\{f_1, \ldots, f_p\}$ linearly independent. In Lemma \ref{faceofKcase} we have described the face $F$ of $S$ generated these states  in the special case where all of the $e_i$ are multiples of each other. In this case $F$ is also a face of $K$. 

 We now remove the restriction that all of the one dimensional subspaces $[e_i]$ coincide.  We are going to partition the set of vectors $e_i \otimes f_i$ into subsets for which these subspaces coincide, and apply Lemma \ref{faceofKcase} to each such subset.  For simplicity of notation, we renumber the vectors in the fashion we now describe.
 
\begin{thm}{thm3}{
Let $e_1, e_2, \ldots, e_p$  and $f_1, f_2, \ldots, f_p$ be unit vectors in $\C^m$ and $\C^n$ respectively, and with $f_1, \ldots, f_p$ linearly independent.  We assume that the vectors are ordered so that $[e_1] , \ldots, [e_q]$ are distinct, and so that  for $i > q$ each $[e_i]$ equals one of $[e_1], \ldots, [e_q]$.
 For $1 \le i \le q$, let $F_i$ be the face of $S$ generated by the states $\{\omega_{e_j \otimes f_j} \mid [e_j] = [e_i]\} \text{ and $1 \le j \le p$}\}$.  Then each $F_i$ is also a face of $K$, and the face $F$ of $S$ generated by $\{\omega_{e_i \otimes f_i} \mid 1\le i \le p \}$ is the direct convex sum of $F_1, \ldots, F_q$.  Moreover, each $F_i$ is affinely isomorphic to the state space of $\B(L_i)$, where $L_i = e_i\otimes \Span\{f_j \mid [e_i] = [e_j]\}$.
In the special case when  $[e_1], \ldots, [e_p]$ are distinct, then $F$ is the convex hull of $\{\omega_{e_i \otimes f_i}\mid 1 \le i \le p\}$, and $F$ is a simplex. }
\end{thm}

\begin{proof}   By Lemma \ref{faceofKcase}, the face $F_i$ of $S$ is equal to the face of $K$ generated by $\{\omega_{e_j \otimes f_j}\mid [e_j] = [e_i]\}$, and is affinely isomorphic to the state space of $\B(L_i)$.  

We will show $L_1, \ldots, L_q$ are independent (i.e., that $L_1 + L_2 + 
\cdots L_q$ is a vector space direct sum).  For $1\le i \le q$ let $e_i \otimes g_i$ be a nonzero vector in $L_i$.  For $i \not= j$, $g_i$ and $g_j$ are linear combinations of disjoint subsets of $f_1, f_2, \ldots, f_p$, so by independence of $f_1, f_2, \ldots, f_p$, the subset $\{g_1, \ldots, g_q\}$ is independent.  Thus by Lemma \ref{productvectorlemma}, $\{e_1\otimes g_1, \ldots, e_p \otimes g_p\}$ is independent, and hence
 the subspaces $L_1, \ldots, L_q$ are independent. Hence by Lemma \ref{directsumlemma}, the convex hull of the faces $F_i$ is a direct convex sum of those faces.
 
 Finally, we need to show that this convex hull coincides with the face $F$ of $S$.  Extreme points of $F$ are extreme points of $S$, so are pure product states.  Suppose that $\omega_{x\otimes y}$ is a pure product state in $F$. Then $\omega_{x\otimes y}$ is in the face of $K$ generated by $\{\omega_{e_i \otimes f_i}\mid 1 \le i \le p\}$, so $x\otimes y$ is in $\Span \{e_i \otimes f_i \mid 1 \le i \le p\}$. By Lemma \ref{productvectorlemma}, $[x] = [e_j]$ for some $j$, and $y \in \Span\{y_i \mid [e_i] = [e_j]\}$. Hence $\omega_{x\otimes y} \in F_j$. Thus each extreme point of $F$ is in some $F_j$, so $F$ is contained in the convex hull of $\{F_i \mid 1 \le i \le q\}$. Evidently $F$ contains every $F_j$, so this convex hull equals $F$.\end{proof}

In Theorem \ref{thm3} we showed that the face $F$ is the direct convex sum of faces that are affinely isomorphic to state spaces of full matrix algebras.  Convex sets of this type were studied by Vershik (in both finite and infinite dimensions), who called them \emph{block simplexes} \cite{Vershik}. Other examples are provided by state spaces of any finite dimensional C*-algebra.  Our Theorem \ref{thm3} provides new examples of such block simplexes.

\begin{cor}{cor3.9}
{Let $e_1, e_2, \ldots, e_p$  and $f_1, f_2, \ldots, f_p$ be unit vectors in $\C^m$ and $\C^n$ respectively.  We assume that $[e_i] \not= [e_j]$ for $ i \not= j$, and that $f_1, f_2, \ldots, f_p$ are linearly independent. If $\lambda_1, \ldots, \lambda_k$ are nonnegative numbers with sum 1, then the separable state $\omega = \sum_i \lambda_i \omega_{e_i\otimes f_i}$ has a unique representation as a convex combination of pure product states.
}
\end{cor}

\begin{proof} Suppose $\omega$ equals the convex combination $\sum_j \gamma_j \tau_j$ where each $\tau_j$ is a pure product state.  Then each $\tau_j$ is in the face $F$ of $S$ generated by $\omega$. 
 By Theorem \ref{thm3}, $F$ is a simplex, and the extreme points of $F$ are all of the form $\omega_{e_i \otimes f_i}$. Since each $\tau_j$ is a vector state, it is a pure state as well, so each state $\tau_j$ must be an extreme point of $F$, and thus must equal some $\omega_{e_i\otimes f_i}$.  Uniqueness of the representation of $\omega$  follows from the uniqueness of convex decompositions into extreme points of a (finite dimensional) simplex.

\end{proof}

\begin{definition}

A separable state $\omega$ has \emph{length} $k$ if $\omega$ can be expressed as a convex combination of $k$ pure product states and admits no decomposition into fewer than $k$ pure product states.  We denote by  $S_k$ the set of separable states of length at most $k$.
\end{definition}

\begin{definition}

A separable state $\omega$ has a \emph{unique decomposition} if it can be written as a convex combination  of pure product states in just one way   \end{definition}

 By the above result, roughly speaking decompositions of separable states  on $\B(\C^m \otimes \C^n)$ of length $\le \max(m,n)$ generically are unique.  Here's a more precise statement.

Let $k \le \max(m,n)$, and let $V_k$ be the set of states $\omega$ admitting a convex decomposition $\omega = \sum_{i=1}^k \lambda_i \omega_{e_i\otimes f_i}$, where $e_1, \ldots, e_k$ and $f_1, \ldots, f_k$ are unit vectors in $\C^m$ and $\C^n$ respectively, $0 < \lambda_i$ for $1\le i \le k$, $[e_1], \ldots, [e_k]$ are distinct, and $\{f_1, \ldots, f_k\}$ is linearly independent. 

\begin{thm}{cor4}
{Let $m,n > 1$. For a given $k \le \max(m,n)$, the states in $V_k$ have length $k$, and have unique decompositions. The set $V_k$ is open and dense in the set $S_k$ of separable states of length at most $k$. }
\end{thm}

\begin{proof} Without loss of generality, we may assume $m\le n$.  By Corollary \ref{cor3.9}, each $\omega \in V_k$ admits a unique representation as a convex combination of pure product states, and each state in $V_k$ has length $k$.  We will show that $V_k$ is open and dense in $S_k$.

 To prove density, let $\omega\in S_k$ have a convex decomposition $\omega = \sum_{i=1}^k \lambda_i\omega_{x_i \otimes y_i}$. By slightly perturbing the coefficients $\lambda_i$ if necessary, we may assume that $\lambda_i > 0$ for all $i$.
 
 Given $\epsilon > 0$, by perturbing each $x_i$ and $y_i$  if necessary, we can find a second convex combination of pure product states $\omega' = \sum_{i=1}^k \lambda_i\omega_{e_i \otimes f_i}$  with $\|\omega - \omega'\|< \epsilon$, with  $[e_1], \ldots, [e_k]$  distinct, and with $\{f_1,  \ldots, f_k\}$ independent.  (Indeed, to achieve independence we may append unit vectors $y_{k+1},\ldots, y_n$  to the vectors $y_1, \ldots, y_k$ to give the subset $\{y_1, y_2, \ldots, y_n\}$ of $\C^n$, and by small perturbations arrange that the determinant of the matrix with columns $y_1, \ldots, y_n$ is nonzero.)  Thus $V_k$ is dense in $S_k$.

 Let $I_0 = \{(\lambda_1, \lambda_2, \ldots, \lambda_k) \in [0,1]^k\mid \sum_i \lambda_i = 1\}$. Let $U_m$ be the unit sphere of $\C^m$ and $U_n$ the unit sphere of $\C^n$.  Let $X = I_0 \times (U_m)^k \times (U_n)^k$.  Define $\psi:X \to S$ by
$$\psi((\lambda_1, \ldots, \lambda_k), (x_1, \ldots, x_k), (y_1, \ldots, y_k)) = \sum_i \lambda_i \omega_{x_i \otimes y_i}.$$
Note that $\psi$ is continuous, that $X$ is compact with  respect to the product topology, and that $\psi(X) =  S_k$.

Now let $X_0$ be the set $\{((\lambda_1, \ldots, \lambda_k), (x_1, \ldots, x_k), (y_1, \ldots, y_k))\}$ of members of $X$ of such that $[x_1], \ldots, [x_k]$ are distinct, such that $\{y_1, \ldots, y_k\}$ is linearly independent, and such that $\lambda_i>0$ for $1 \le i \le k$. 
By lower semicontinuity of the rank of a matrix whose columns are $y_1, \ldots, y_k$, 
the set of elements $((\lambda_1, \ldots, \lambda_k), (x_1, \ldots, x_m), (y_1, \ldots, y_k))$ of $X$ with $\{y_1, \ldots, y_k\}$ linearly independent is open in $X$, so it is clear that $X_0$ is an open subset of $X$.  By construction, $\psi(X_0 )= V_k$.  Since $X_0$ is open in $X$, then $X\setminus X_0$ is closed and hence compact. Since  $\psi$ maps $X \setminus X_0$ onto $S_k \setminus \psi(X_0)$, then the latter is closed, so $V_k = \psi(X_0)$ is open in $S_k$. 
\end{proof}

 As remarked in the introduction,  the sets $V_k$ are open and dense in the relative topology on $S_k$, but are not open or dense in $S$ or $K$ if $mn > 1$. To see this recall that a point $\sigma$ in a convex set $C$ is an algebraic interior point if for every point $\rho$ in $C$ there is a point $\tau$ in $C$ such that $\sigma$ lies on the open line segment from $\rho$ to $\tau$.  Clearly for every algebraic interior point $\sigma$  of $S$ and every pure product state $\rho$, there is a convex decomposition of $\sigma$ that includes $\rho$ with positive weight. Since there are infinitely many pure product states, there are infinitely many convex decompositions for every algebraic interior point of $S$.

  Every nonempty subset  which is open in $S$ contains an algebraic interior point of $S$ (\cite[pp. 88-91]{SW}), so contains points with nonunique decompositions.  Thus $V_k$ is not open in $S$ or $K$.  It is not dense in $S$ or $K$, since for any $m, n$ there exists $r > 0$ such that all states $\sigma$ within a distance $r$ from  the normalized tracial state are separable, cf. \cite[Thm. 1]{Z}.  Every such state $\sigma$ is an algebraic interior point of $S$, and so fails to have a unique decomposition.

Observe that  Theorem \ref{cor4} implies that $V_k$ is also open and dense in the set of separable states of length equal to $k$.

\section{Description of convex decompositions}

Let $e_1, e_2, \ldots, e_p$  and $f_1, f_2, \ldots, f_p$ be unit vectors in $\C^m$ and $\C^n$ respectively, with $f_1, \ldots, f_p$ linearly independent. Suppose $\omega$ is a convex combination of $\{\omega_{e_i \otimes f_i} \mid 1 \le i \le p\}$.
 In this section, we will describe all convex decompositions of $\omega$ into pure product states.

Let $\omega = \sum_i \lambda_i \omega_i$ be any convex decomposition of $\omega$ into pure product states.  Then following the notation of Theorem \ref{thm3}, each $\omega_i$ is in $\face_S(\omega) \subset F$.   Since each $\omega_i$ is an extreme point of $S$, and $F$ is the direct convex sum of the faces $F_i$, then each $\omega_i$ must be in some $F_k$.  If we define $\gamma_k = \sum_{\{i \mid \omega_i \in F_k\}}\lambda_i$ and $\sigma_k = \gamma_k^{-1}\sum_{\{i \mid \omega_i \in F_k\}}\lambda_i\omega_i$, then $\omega$ has the convex decomposition 
\begin{equation}\label{decomp}
\omega = \sum_k \gamma_k \sigma_k \text{ with  $\sigma_k \in F_k$ for each $k$}.
\end{equation}
 Since $F$ is the direct convex sum of the $F_k$, the decomposition of $\omega$ in \eqref{decomp} is unique.  

All possible convex decompositions of $\omega$ into pure product states can be found by starting with the unique decomposition $\omega = \sum_k \gamma_k \sigma_k$ with $\sigma_k \in F_k$, and then decomposing each $\sigma_k$ into pure states. (Every state in $F_k$ is separable, so pure states are pure product states). Since $F_k$ is affinely isomorphic to the state space of $\B(L_k)$, unless each $\sigma_k$ is itself a pure state, this can be done in many ways, as we discussed in the introduction. The possibilities have been described in \cite{Wootters, Schrodinger, Kirkpatrick}.

A decomposition of a separable state $\omega$ as a convex combination of pure product states  
can be interpreted as a representation of $\omega$ as the barycenter of a probability measure 
on the extreme points of $S$. With this interpretation the statement above can be rephrased 
in terms of the concept of dilation of measures (as defined e.g. in \cite[p. 25]{Alfsen}). If $\omega$ is
given as above, then the probability measures on pure product states that represent $\omega$
are precisely those which are dilations of the uniquely determined probability measure  
$\mu = \sum_k \gamma_k\mu_k$ obtained from (6) with $\mu_k=\delta_{\sigma_k}$.

\section{Affine automorphisms of the space $S$ of separable states}

\begin{notation}
Fix $m, n$. We denote  the state space of $\B(\C^m)$ by $K_m$, the state space of $\B(\C^n)$ by $K_n$, and the state space of $\B(\C^m \otimes \C^n)$ by $K$ or $K_{m,n}$.  The convex set of separable states in $K$ is denoted by $S$ or $S_{m,n}$.  We will sometimes deal with  a second algebra $\B(\C^{m'} \otimes \C^{n'})$, whose state space and separable state spaces we will denote by $K'$ or $S'$ respectively.

\end{notation}

From Theorem \ref{thm3}, the face of $S$ generated by two distinct pure product states $\omega_1 \otimes \sigma_1$ and $\omega_2 \otimes \sigma_2$ is a line segment (if $\omega_1 \not= \omega_2$ and $\sigma_1 \not= \sigma_2)$ or is isomorphic to the state space of $\B(\C^2)$ and hence is a 3-ball (when $\omega_1 = \omega_2$ but $\sigma_1 \not= \sigma_2$, or  when $\sigma_1 = \sigma_2$ but $\omega_1 \not= \omega_2$). (By a 3-ball we mean a convex set affinely isomorphic to the closed unit ball of $\R^3$.  The fact that the state space of $\B(\C^2)$ is a 3-ball can be found in many places, e.g., \cite[Thm. 4.4]{Alfsen-Shultz}.)

  We define a relation $R$ on the pure product states of $K$ by $\rho \RR \tau$ if $\face_S(\rho, \tau)$ is a 3-ball. By the remarks above, $(\omega_1 \otimes \sigma_1) \RR (\omega_2 \otimes \sigma_2)$ iff ($\omega_1 = \omega_2$ but $\sigma_1 \not= \sigma_2$) or  ($\sigma_1 = \sigma_2$ but $\omega_1 \not= \omega_2$). Note that an affine isomorphism $\Phi:S\to S'$ will take faces of $S$ to faces of $S'$, and will take 3-balls to 3-balls, so for pure product states $\rho, \tau$ we have  $\rho \RR \tau$ iff $\Phi(\rho) \RR \Phi(\tau)$.

The idea of the following lemmas is to show that if $\Phi(\omega\otimes \sigma) = \phi(\omega, \sigma) \otimes \psi(\omega, \sigma)$, then $\phi$ depends only on the first argument and $\psi$ depends only on the second argument, or possibly vice versa. Although we are interested in affine automorphisms of a single space of separable states, it will be easier to establish the needed lemmas in the context of affine isomorphisms  from $S$ to $S'$.

We use the notation $\partial_eC$ for the set of extreme points of a convex set $C$. For example, $\partial_eK$ is the set of pure states on $\B(\C^m \otimes \C^n)$.

\begin{lemma}\label{foureqns}  Let $\Phi:S_{m,n}\to S_{m', n'}$ be an affine isomorphism.  Let $\omega_1$, $\omega_2 $ be distinct pure states in $K_m$ and $\sigma_1$, $\sigma_2$ distinct pure states in $K_n$.    Then the following four equations cannot hold simultaneously.
\begin{align}\label{four}
\Phi(\omega_1 \otimes \sigma_1) &= \rho_1 \otimes \tau_1\cr
\Phi(\omega_1 \otimes \sigma_2) &= \rho_1 \otimes \tau_2\cr
\Phi(\omega_2 \otimes \sigma_1) &= \rho_2 \otimes \tau_3\cr
\Phi(\omega_2 \otimes \sigma_2) &= \rho_3\otimes \tau_3\cr
\end{align}
for $\rho_1, \rho_2, \rho_3 \in \partial_e K_{m'}$ and $\tau_1, \tau_2, \tau_3 \in \partial_e K_{n'}$.
\end{lemma}

\begin{proof} We assume for contradiction that all four equations hold.
Since $(\omega_1 \otimes \sigma_1) \RR (\omega_2 \otimes \sigma_1)$, then $(\rho_1 \otimes \tau_1) \RR (\rho_2 \otimes \tau_3)$. Hence
\begin{equation}\label{e1}
\rho_1 = \rho_2 \text{ or  } \tau_1 = \tau_3.
\end{equation}
Similarly
$(\omega_1 \otimes \sigma_2) \RR (\omega_2 \otimes \sigma_2)$, so $(\rho_1 \otimes \tau_2) \RR (\rho_3 \otimes \tau_3)$. Hence
\begin{equation}\label{e2}
\rho_1 = \rho_3 \text{ or  } \tau_2 = \tau_3.
\end{equation}
Since we are assuming that $\omega_1 \not= \omega_2$ and $\sigma_1 \not= \sigma_2$, the four states $\{\omega_i \otimes \sigma_j \mid 1\le i, j \le 2\}$ are distinct, so the four states on the right side of \eqref{four} must be distinct. Combining \eqref{e1} and \eqref{e2} gives four possibilities, each contradicting the fact that the states on the right side of \eqref{four} are distinct. Indeed:
\begin{align*}
(\rho_1 = \rho_2 \text{ and } \rho_1 = \rho_3) \implies & \rho_2 \otimes \tau_3 = \rho_3 \otimes \tau_3\cr
(\rho_1 = \rho_2 \text{ and } \tau_2 = \tau_3) \implies & \rho_1 \otimes \tau_2 = \rho_2 \otimes \tau_3\cr
(\tau_1 = \tau_3 \text{ and } \rho_1 = \rho_3) \implies & \rho_1 \otimes \tau_1 = \rho_3 \otimes \tau_3\cr
(\tau_1 = \tau_3 \text{ and } \tau_2 = \tau_3) \implies & \rho_1 \otimes \tau_1 = \rho_1 \otimes \tau_2.\cr
\end{align*}
We conclude that the four equations in \eqref{four} cannot hold simultaneously.
\end{proof}

\begin{definition} Recall that we identify $\B(\C^m \otimes \C^n)$ with $\B(\C^m) \otimes \B(\C^n)$. The \emph{swap isomorphism} $(\alpha_{m,n})_*: \B(\C^n \otimes \C^m)\to \B(\C^m \otimes \C^n)$ is the *-isomorphism that satisfies $(\alpha_{m,n})_*(A \otimes B) = B \otimes A$.  If operators in $\B(\C^m \otimes \C^n)$ are identified with matrices, the swap isomorphism is the same as the ``canonical shuffle" discussed in \cite[Chapter 8]{Paulsen}. The dual map $\alpha_{m,n}$ is an affine isomorphism from the state space of  $\B(\C^m \otimes \C^n)$ to the state space of $ \B(\C^n \otimes \C^m)$, with $\alpha_{m,n}(\omega \otimes \sigma) = \sigma \otimes \omega$.  This restricts to an affine isomorphism from $S_{m,n}$ to $S_{n,m}$, which we also refer to as the swap isomorphism.  If $m = n$, then $(\alpha_{m,m})_*$ is a *-automorphism of $ \B(\C^m \otimes \C^m)$, $\alpha_{m,m}$ is an affine automorphism of the state space $K$, and restricts to an affine automorphism of the space $S$ of separable states.
\end{definition}

\begin{lemma}\label{factor}  Let $\Phi:S_{m,n}\to S_{m', n'}$ be an affine isomorphism.  At least one of the following two possibilities occurs:
\begin{enumerate}
\item[(i)] For every $\omega \in \partial_eK_m$ there exists $\rho \in \partial_e K_{m'}$ such that 
$\Phi(\omega \otimes K_n) = \rho \otimes K_{n'}$, and for every $\sigma \in \partial_e K_n$ there exists $\tau \in \partial_e K_{n'}$ such that $\Phi(K_m\otimes \sigma) = K_{m'} \otimes \tau$.
\item[(ii)] For each $\omega \in \partial_eK_m$ there exists $\tau \in \partial_e K_{n'}$ such that
$\Phi(\omega \otimes K_n) = K_{m'} \otimes \tau$, and  for every $\sigma \in \partial_e K_n$ there exists $\rho \in \partial_e K_{m'}$ such that $\Phi(K_m\otimes \sigma) = \rho \otimes K_{n'}$.
\end{enumerate}
If (i) occurs, then $m = m'$ and $n = n'$.  If (ii) occurs, then $m = n'$ and $n = m'$.
\end{lemma}

\begin{proof}
 For fixed $\omega\in \partial_eK_m$  and distinct $\sigma_1, \sigma_2 \in \partial_e K_n$ we have $(\omega\otimes \sigma_1) \RR (\omega \otimes \sigma_2)$,  so $\Phi(\omega\otimes \sigma_1) \RR \Phi(\omega \otimes \sigma_2)$.  Thus either there exist $\rho_1 \in \partial_eK_{m'}$ and distinct $\tau_1, \tau_2 \in \partial_e K_{n'}$ such that
\begin{equation}\label{eq3}
\Phi(\omega\otimes \sigma_i) = \rho_1 \otimes \tau_i \text{ for $i = 1, 2$},
\end{equation}
or  there exist distinct $\rho_1, \rho_2 \in \partial_e K_{m'}$ and $\tau_3 \in \partial K_{n'}$ such that
\begin{equation}\label{eq4}
\Phi(\omega\otimes \sigma_i) = \rho_i \otimes \tau_3 \text{ for $i = 1, 2$}.
\end{equation}
We will show that \eqref{eq3} implies (i), and \eqref{eq4} implies (ii).

Suppose that \eqref{eq3} holds. Let $\sigma \in \partial_e K_n$ with $\sigma \not= \sigma_1$ and $\sigma\not=\sigma_2$, and let $\Phi(\omega\otimes \sigma) = \rho \otimes \tau$.   Since $(\omega\otimes \sigma) \RR (\omega\otimes \sigma_i)$ for $i = 1, 2$, then $(\rho\otimes \tau) \RR (\rho_1 \otimes \tau_i)$ for $i = 1, 2$.  Hence ($\rho = \rho_1$ or $\tau = \tau_1$) and ($\rho = \rho_1$ or $\tau = \tau_2$). Since $\tau_1 \not= \tau_2$, then $\rho = \rho_1$.  It follows that $\Phi(\omega \otimes K_n) \subset \rho_1 \otimes K_{n'}$.  Thus
\begin{equation}
\Phi(\omega\otimes \sigma_i) = \rho_1 \otimes \tau_i \text{ for $i = 1, 2$}\implies \Phi(\omega \otimes K_n)\subset \rho_1\otimes K_{n'}.\label{imp}
\end{equation}
Now \eqref{eq3} also implies 
\begin{equation}\label{eq5}
\Phi^{-1}(\rho_1 \otimes \tau_i) = \omega\otimes \sigma_i \text{ for $i = 1, 2$}.
\end{equation}
If \eqref{eq3} holds (and hence also \eqref{eq5}, then applying the implication \eqref{imp} to \eqref{eq5} with $\Phi^{-1}$ in place of $\Phi$  shows 
$ \Phi^{-1}(\rho_1\otimes K_{n'})
 \subset \omega \otimes K_n$, so by \eqref{imp}  equality holds.  Hence we have shown
\begin{equation}
\Phi(\omega\otimes \sigma_i) = \rho_1 \otimes \tau_i \text{ for $i = 1, 2$}\implies \Phi(\omega \otimes K_n)= \rho_1\otimes K_{n'}.\label{imp3}
\end{equation}

Now suppose instead that \eqref{eq4} holds.  Let $\alpha_{m', n'}$ be the swap affine isomorphism defined above, so that $\alpha_{m', n'}: S_{m', n'} \to S_{n', m'}$. Then
\begin{equation}
(\alpha_{m', n'}\circ\Phi)(\omega\otimes \sigma_i) = \alpha_{m', n'}(\rho_i \otimes \tau_3) = \tau_3 \otimes \rho_i \text{ for $i = 1, 2$}.
\end{equation}
By the implication \eqref{imp3} applied to $\alpha_{m', n'} \circ \Phi$ we conclude that 
$$(\alpha_{m', n'} \circ \Phi)(\omega \otimes K_n)= \tau_3\otimes K_{m'},$$ so $$\Phi(\omega \otimes K_n)= \alpha_{m', n'}^{-1}(\tau_3\otimes K_{m'}) = K_{m'} \otimes \tau_3.$$ Thus we have proven the implication 
\begin{equation}
\Phi(\omega\otimes \sigma_i) = \rho_i \otimes \tau_3 \text{ for $i = 1, 2$} \implies \Phi(\omega \otimes K_n)= K_{m'} \otimes \tau_3.\label{eq16}
\end{equation}

By Lemma \ref{foureqns} and the implications \eqref{imp3} and \eqref{eq16}, either \eqref{eq3} must hold for all $\omega \in \partial_eK_m$ or \eqref{eq4} must hold for all $\omega \in \partial_eK_m$.  We conclude that either
\begin{equation}\label{choice1}
\forall  \omega \in \partial_eK_m \quad\exists \rho \in \partial_e K_{m'} \text{ such that }
\Phi(\omega \otimes K_n) = \rho \otimes K_{n'}
\end{equation}
or
\begin{equation}\label{choice2}
\forall  \omega \in \partial_eK_m \quad\exists \tau \in \partial_e K_{n'} \text{ such that }
\Phi(\omega \otimes K_n) = K_{m'} \otimes \tau.
\end{equation}
Similarly, either
\begin{equation}\label{choice3}
\forall  \sigma \in \partial_eK_n \quad\exists \tau' \in \partial_e K_{n'} \text{ such that }
\Phi(K_m\otimes \sigma) = K_{m'}\otimes \tau'
\end{equation}
or
\begin{equation}\label{choice4}
\forall  \sigma \in \partial_eK_n \quad\exists \rho' \in \partial_e K_{m'} \text{ such that }
\Phi( K_m\otimes \sigma) =  \rho'\otimes K_{n'}. 
\end{equation}
Suppose that \eqref{choice1} and \eqref{choice4} both held.  For $\omega \in K_m$ and $\sigma \in K_n$ note that $\omega \otimes \sigma$ is in both $\omega \otimes K_n$ and $K_m \otimes \sigma$, so $\rho \otimes K_{n'}$ and $\rho' \otimes K_{n'}$ are not disjoint.  This implies $\rho = \rho'$, so $\Phi(\omega\otimes K_n) = \Phi(K_m \otimes \sigma)$.  Since $\Phi$ is bijective,  $\omega \otimes K_n = K_m \otimes \sigma$ follows.  This is possible only if $m = n = 1$.  If $m = n = 1$, then all of \eqref{choice1}, \eqref{choice2}, \eqref{choice3}, \eqref{choice4} hold. Similarly if \eqref{choice2} and \eqref{choice3} both held then $m = n = 1$ is again forced.  Thus the possibilities  are that \eqref{choice1} and \eqref{choice3} both hold (which is the same as statement (i) of the lemma), or that \eqref{choice2} and \eqref{choice4} hold (equivalent to (ii)), or that $m = n = 1$, in which case both (i) and (ii) hold.

Finally, since the affine dimensions of $K_p$ and $K_q$ are different when $p \not= q$,  the statement in the last sentence of the lemma follows.

\end{proof}

If $\psi_1:K_m\to K_m$ and $\psi_2:K_n \to K_n$ are affine automorphisms, then we can extend each to linear maps on the linear span, and form the tensor product $\psi_1 \otimes \psi_2$. This will be bijective, but not necessarily positive.  (A well known example of this phenomenon occurs when $\psi_1$ is the identity map and $\psi_2$ is the transpose map.) However,  $\psi_1$ and $\psi_2$ will map pure states to pure states, and hence $\psi_1 \otimes \psi_2$ will map pure product states to pure product states.  Thus $\psi_1 \otimes \psi_2$ will map $S$ onto $S$, and hence will be an affine automorphism of $S$. We will now see that all affine automorphisms of $S$ are either such a tensor product  of automorphisms or such a tensor product composed with the swap automorphism.

\begin{theorem}\label{factoring}
If $m \not= n$, and $\Phi:S\to S$ is an affine automorphism, then there exist unique affine automorphisms $\psi_1:K_m\to K_m$ and $\psi_2: K_n \to K_n$ such that $\Phi = \psi_1 \otimes \psi_2$.  If $m =n$ then either we can write $\Phi =   (\psi_1 \otimes \psi_2)$ or $\Phi = \alpha_{m,m} \circ  (\psi_1 \otimes \psi_2)$, where $\psi_1, \psi_2$ are again unique affine automorphisms  of $K_m$ and $K_n$ respectively, and $\alpha_{m,m}:S \to S$ is the swap automorphism.
\end{theorem}

\begin{proof}  We apply Lemma \ref{factor}. For each $\omega \in \partial_eK_m$ and $\sigma \in \partial_eK_n$, define $\phi_\sigma:K_m \to K_m$ and $\psi_\omega: K_n \to K_n$ by
$$\Phi(\omega \otimes \sigma) = \phi_\sigma(\omega) \otimes \psi_\omega(\sigma).$$
Suppose first that case (i) of Lemma \ref{factor} occurs.  Then $\psi_\sigma(\omega)$ is independent of $\sigma$ and $\psi_\omega(\sigma) $ is independent of $\omega$. Therefore there are functions $\psi_1:K_m \to K_m$ and $\psi_2: K_n \to K_n$ such that 
$$\Phi(\omega \otimes \sigma) = \psi_1(\omega) \otimes \psi_2(\sigma).$$
Since $\Phi$ is bijective and affine, so are $\psi_1$ and $\psi_2$.

Suppose instead that case (ii) of Lemma \ref{factor} occurs. Then $m = n$. If we define $\Phi'=\alpha_{m,m}\circ \Phi$, then $\Phi': S\to S$ satisfies case (i) of Lemma \ref{factor}.  Then from the first paragraph we can choose affine automorphisms $\psi_1:K_m \to K_m$ and $\psi_2: K_n \to K_n$ such that $\Phi' = \psi_1 \otimes \psi_2$.  Since $\alpha_{m,m}^2 $ is the identity map, then $\Phi = \alpha_{m,m} \circ (\psi_1 \otimes \psi_2)$.
\end{proof}

We review some well known facts about affine automorphisms of state spaces and maps on the underlying algebra. Let $\psi$ be an affine automorphism of $K_m$. Then $\psi$ extends uniquely to a linear map on the linear span of $K_m$, which we also denote by $\psi$, and this map is the dual of a unique linear map $\psi_*$ on $\B(\C^m)$. 
By a result of Kadison \cite{Kad-Isom} $\psi_*$ will be a *-isomorphism or a *-anti-isomorphism.  (Since the restriction of an affine automorphism to pure states  preserves transition probabilities, this also follows from Wigner's theorem \cite{Wigner}).   The map $\psi_*$ will be a *-isomorphism iff $\psi_*$  is completely positive, which is equivalent to $\psi$ being completely positive.     If $\psi_*$ is a *-isomorphism, then $\psi_*$ is implemented by a unitary, i.e.,   there is a unitary $U\in \B(\C^m)$ such that $\psi_*(A) = UAU^*$.

If $\psi_*$ is a *-anti-isomorphism, then the composition of $\psi_*$ with the transpose map  (in either order) gives a *-isomorphism, and the map $\psi_*$ is completely copositive.     It follows that an affine automorphism $\psi$ of $K_m$ is either completely positive or completely copositive, and $\psi $ is completely positive iff $\psi ^{-1}$ is completely positive.  
If $t$ denotes the transpose map on $\B(\C^m)$ or $\B(\C^n)$, then $t$ is positive but  $t \otimes id$ and $id \otimes t$ are not positive on $\B(\C^m) \otimes \B(\C^n)$ if $m, n > 1$.
Background can be found in  \cite[Chapters 4, 5]{Alfsen-Shultz}.

Recall that a \emph{local unitary} in $\B(\C^m \otimes \C^n)$ is a tensor product $U_1 \otimes U_2$ of unitaries.

\begin{theorem}
Every affine automorphism of the space $S$ of separable states on $\B(\C^m \otimes \C^n)$ is the dual of   conjugation by local unitaries,  one of the two partial transpose maps, the swap map (if $m = n$), or a composition of these maps.   An affine automorphism $\Phi$ of $S$ extends uniquely to an affine automorphism of the full state space $K$ iff it can be expressed as one of the compositions just mentioned with  both or neither of the partial transpose maps involved.  
\end{theorem}

\begin{proof}  We note first that if $m = 1$ or $n = 1$, the result is clear, so we assume hereafter that $m \ge 2$ and $n \ge 2$. 

We next show that if $\psi_1:K_m \to K_m$ and $\psi_2:K_n \to K_n$ are affine automorphisms, then $\Phi = \psi_1 \otimes \psi_2$ is an affine automorphism of $K$ iff $\psi_1$ and $\psi_2$ are both completely positive or both completely copositive. 

If $\psi_1$ and $\psi_2$ are completely positive, then $\Phi = \psi_1\otimes \psi_2 = (id \otimes \psi_2) \circ (\psi_1 \otimes id)$ is positive; hence $\Phi(K) \subset K$.  Furthermore, $\psi_1^{-1}$ and $\psi_2^{-1}$ will be completely positive, so $\Phi^{-1}$ is positive, and hence $\Phi(K) = K$.   If $\psi_1$ and $\psi_2$ are completely copositive, then $(t\circ \psi_1) \otimes (t\circ \psi_2)$ is positive. Composing with $t \otimes t$ shows $\psi_1 \otimes \psi_2$ is positive and as above we conclude that $\Phi(K) = K$.  On the other hand, if $\psi_1$ is completely positive and $\psi_2$ is completely copositive, then $\psi_1 \otimes (t\circ \psi_2)$ is positive, so $(id \otimes t)\circ (\psi_1 \otimes \psi_2)$ is positive.  If  $(\psi_1 \otimes \psi_2)(K)=K$, then $id\otimes t$ would be positive, a contradiction since $m, n \ge 2$.  Thus in this case $\psi_1 \otimes \psi_2$ is not an affine automorphism of $K$.
 
 If $\psi_1$ and $\psi_2$ are completely positive, then they are implemented by unitaries, so $\Phi = \psi_1\otimes \psi_2$ is implemented by a local unitary.  If both are completely copositive, then $t\circ \psi_1$ and $t\circ \psi_2$ are implemented by unitaries, so $(t\otimes t) \circ (\psi_1 \otimes \psi_2)$ is implemented by a local unitary.  Then $\Phi = (t\otimes t) \circ (t\otimes t) \circ (\psi_1 \otimes \psi_2)$ is the composition of the transpose map on $K$ and conjugation by local unitaries.

 The first statement of the theorem now follows from Theorem \ref{factoring}. Uniqueness follows from the fact that the linear span of $S$ contains $K$.
\end{proof}

\begin{definition} Let $\Phi:K\to K$ be an affine automorphism.  We say $\Phi$ \emph{preserves separability} if $\Phi$ takes separable states to separable states, i.e., if $\Phi(S) \subset S$.  A state $\omega$ in $K$ is \emph{entangled} if $\omega$ is not separable.  $\Phi$ \emph{preserves entanglement} if $\Phi$ takes entangled states to entangled states.
\end{definition}

\begin{corollary} Let $\Phi: K_{m,n} \to K_{m,n}$ be an affine automorphism. Then $\Phi$  preserves entanglement and separability iff $\Phi$ is a composition of maps of the types (i)   conjugation by local unitaries, (ii) the transpose map, (iii) the swap automorphism (in the case that $m = n$).
\end{corollary}

\begin{proof}
If $\Phi$  preserves entanglement and separability, then $\Phi$ maps $S$ into $S$ and $K\setminus S$ into $K\setminus S$, which is equivalent to $\Phi(S) = S$. 
\end{proof}

\begin{corollary}\label{one} If $\Phi_t: S \to S$ is a one-parameter  group of  affine automorphisms, then there are one-parameter groups of unitaries $U_t$ and $V_t$ such that 
$\Phi_t(\omega(A)) = \omega((U_t \otimes V_t)A(U_t^* \otimes V_t^*))$.
\end{corollary}

\begin{proof}
For each $t$, factor $\Phi_t = \phi_t \otimes \psi_t$ or $\Phi_t = \alpha \circ (\phi_t \otimes \psi_t)$, where $\alpha$ is the swap automorphism.  In the latter case, 
$$\Phi_{2t} = \Phi_t \circ \Phi_t = \alpha \circ (\phi_t \otimes \psi_t) \circ \alpha \circ (\phi_t \otimes \psi_t)$$
$$= (\phi_t \otimes \psi_t) \circ (\phi_t \otimes \psi_t) = (\phi_t\circ \phi_t) \otimes (\psi_t \circ \psi_t).$$
It follows that the swap automorphism is not needed for $\Phi_{2t}$, and hence for $\Phi_t$ for any $t$.  Uniqueness of the factorization $\Phi_t = \phi_t \otimes \psi_t$ shows that $\phi_t$ and $\psi_t$ are also one parameter groups of affine automorphisms. By a result of Kadison \cite{Kad}, such automorphisms are implemented by one parameter groups of unitaries.

\end{proof}

\begin{corollary} If $\Phi_t: K \to K$ is a one-parameter  group of  entanglement preserving affine automorphisms, then there are one-parameter groups of unitaries $U_t$ and $V_t$ such that 
$\Phi_t(\omega(A)) = \omega((U_t \otimes V_t)A(U_t^* \otimes V_t^*))$.
\end{corollary}

\begin{proof} Since $\Phi_t$ and $(\Phi_t)^{-1} = \Phi_{-t}$ preserve entanglement, then $\Phi_t$ maps $S$ onto $S$, so this corollary follows from Corollary \ref{one}.

\end{proof}

\end{document}